\documentclass[final,english,12pt]{article}
\usepackage{amsfonts,amsmath,amsthm,amscd,amssymb,latexsym,cite}

    \usepackage[T2A]{fontenc}
    \usepackage[cp1251]{inputenc}
    \usepackage[russian]{babel}

\usepackage[notref,notcite]{showkeys}
\usepackage{amsfonts}
\usepackage{amssymb}
\usepackage{stmaryrd}
\usepackage{textcomp}
\usepackage{amsthm}
\usepackage{amsmath}

\usepackage{ifpdf}
\ifpdf
\usepackage[pdftex]{graphicx}
\usepackage[pdftex,dvipsnames,usenames]{color}
\else
\usepackage{graphicx}
\usepackage[dvips,dvipsnames,usenames]{color}
\fi
 \DeclareMathOperator{\Sp}{Sp}
\DeclareMathOperator{\diam}{diam}
\DeclareMathOperator{\is}{is}

\newtheorem{theorem}{Теорема}
\newtheorem{lemma}[theorem]{Лемма}
\newtheorem{definition}[theorem]{Определение}
\newtheorem{corollary}[theorem]{Следствие}
\newtheorem{statement}[theorem]{Утверждение}

\theoremstyle{definition}
\newtheorem{remark}[theorem]{Замечание}
\newtheorem{example}[theorem]{Пример}
\begin{document}
\begin{center}{\LARGE
On the Gomori-Hu inequality}
\\\bigskip
{\large E. Petrov and O. Dovgoshey}
\bigskip
\end{center}
\begin{center}
\textbf{Abstract}
\end{center}
It was proved by Gomori and Hu in 1961 that for every finite nonempty ultrametric space $(X,d)$ the following inequality $|\Sp(X)|\leqslant |X|-1$ holds with $\Sp(X)=\{d(x,y):x,y \in  X, x\neq y\}$.
We characterize the spaces $X$, for which the equality in this inequality is attained by the structural properties of some  graphs and show that the set of isometric types of such $X$ is dense in the Gromov-Hausdorff space   of the compact ultrametric spaces.

\bigskip

\bigskip

\begin{center}{\LARGE
О неравенстве Гомори - Ху}
\\\bigskip
{\large Е. А. Петров и А. А. Довгошей}
\end{center}\bigskip

\begin{abstract}
В 1961 году Гомори и Ху доказали,  что для любого конечного непустого ультраметрического пространства $(X,d)$ выполняется неравенство $|\Sp(X)|\leqslant |X|-1$, где $\Sp(X)=\{d(x,y):x,y \in  X, x\neq y\}$. Мы  характеризуем пространства $X$, для которых достигается равенство, посредством структурных свойств  некоторых графов и показываем, что множество типов изометрий таких $X$ плотно в пространстве Громова-Хаусдорфа изометрических типов компактных ультраметрических пространств.
\textbf{Ключевые слова:} конечное ультраметрическое пространство, строго бинарное дерево, полный двудольный граф, спектр ультраметрического пространства,  отображение сохраняющее шары, метрика Громова-Хаусдорфа.
\end{abstract}

\section{Введение}
В 1961 году Гомори (Gomori)  и Ху (Hu) ~\cite{HomoriHu(1961)}, исследуя потоки в сетях, установили неравенство, которое на языке ультраметрических пространств может быть сформулировано следующим образом: если $(X,d)$~--- конечное непустое ультраметрическое пространство со спектром
$$
\Sp(X):=\{d(x,y):x,y \in  X, x\neq y\},
$$
то
\begin{equation}\label{eq1}
|\Sp(X)|\leqslant |X|-1.
\end{equation}
В ~\cite{HomoriHu(1961)} неравенство ~(\ref{eq1}) было получено при изучении остовных деревьев наименьшего веса в полных взвешенных графах. В настоящей работе после элементарного доказательства неравенства Гомори-Ху (теорема ~\ref{th1.3}) мы исследуем семейство $\mathfrak{U}$ конечных ультраметрических пространств $X$, для которых неравенство Гомори-Ху обращается в равенство. В частности:
 \begin{itemize}
   \item Описаны структурные свойства некоторых графов, связанных с конечными ультраметрическими пространствами $X$, логически эквивалентные принадлежности $X\in \mathfrak{U}$ (теоремы~\ref{th7*}  и ~\ref{th2.1}).
   \item Установлено, что корневые деревья, представляющие ультраметрические пространства $X,Y\in \mathfrak U$, изоморфны тогда и только тогда, когда существует сохраняющая шары биекция $F:X\to Y$  (теорема~\ref{th13*}).
   \item Показано, что классы изометрии пространств из $\mathfrak  U$ всюду плотны в пространстве Громова-Хаусдорфа, состоящем из классов изометрии компактных ультраметрических пространств (теорема ~\ref{th4.2}).
 \end{itemize}
 Напомним необходимые определения. \textit{Ультраметрикой}  на множестве $X$ называется функция \mbox{$d:X\times X\rightarrow \mathbb{R}^+$}, $\mathbb{R}^+=[0,\infty)$, такая, что при всех $x,y,z \in X$ выполнены соотношения:
\begin{itemize}
\item [(i)] $d(x,y)=d(y,x)$,
\item [(ii)] $(d(x,y)=0)\Leftrightarrow (x=y)$,
\item [(iii)] $d(x,y)\leq \max \{d(x,z),d(z,y)\}$.
\end{itemize}
Пара $(X,d)$ называется \emph{ультраметрическим пространством}. Неравенство (iii) часто называют \emph{сильным неравенством треугольника}. Функция $d:X\times X\rightarrow \mathbb{R}^+$, удовлетворяющая обычному неравенству треугольника $d(x,y)\leqslant d(x,z)+d(z,y)$ и свойствам (i)-(ii), называется {\it метрикой}, а пара $(X,d)$~--- \emph{метрическим пространством}. \emph{Диаметр} метрического пространства есть величина
$$
\diam X:=\sup\{d(x,y):x,y \in X\}.
$$
Двухэлементное подмножество $\{a,b\}$ множества $X$ является \emph{диаметральной парой} точек для $(X,d)$, если $d(a,b)=\diam X$. Мы будем рассматривать только те пространства $(X,d)$, для которых $X\neq \varnothing$.

Под \textit{графом} мы  понимаем пару $(V,E)$, состоящую из непустого множества $V$ и (возможно пустого) множества $E$, элементы которого есть неупорядоченные пары различных точек из $V$. Для графа $G=(V,E)$ множество $V=V(G)$ называется \textit{множеством вершин}, а $E=E(G)$ ~--- \textit{множеством рёбер}.

 Граф $H$ является {\it подграфом} графа $G$, $H\subseteq G$, если $V(H)\subseteq V(G)$ и $E(H)\subseteq E(G)$. Граф $G$ {\it конечен}, если $|V(G)|<\infty$. Если $E(G)=\varnothing$, то $G$~--- \textit{пустой граф}. Конечный непустой граф $P\subseteq G$ называется \textit{путём} (в $G$), если вершины из $P$ можно без повторений занумеровать в последовательность $(v_1,v_2,...,v_n)$ так, что $ (\{v_i,v_j\}\in E(P))\Leftrightarrow (|i-j|=1)$. Две вершины в графе \textit{связаны}, если существует соединяющий их путь. {\it Связный граф} — граф, в котором все вершины связаны.

\textit{Деревом} называется связный граф, не имеющий циклов.  Выбранная вершина дерева называется \emph{корнем дерева}. Дерево, содержащее такую вершину, называется \emph{корневым деревом}. Вершину дерева иногда называют \emph{узлом.} \emph{Уровень узла} — длина пути от корня до узла, $m$-й \emph{ярус} дерева~--- множество узлов дерева, на уровне $m$ от корня дерева. \emph{Потомками} данного узла будем называть все узлы последующего яруса, смежные с данным узлом. \emph{Лист} дерева~--- вершина дерева инцидентная с единственным ребром. \emph{$m$-арное дерево}~--- это дерево, в котором степени вершин не превосходят $m+1$. \emph{Внутренний узел}~--- узел дерева не являющийся листом. \emph{Помеченный граф} — граф, вершинам которого присвоены какие-либо метки, например, числа или символы какого-нибудь алфавита.

\emph{Строго бинарным деревом} (см. \cite[стр. 298]{Finch(2003)}) называется корневое дерево у которого корень  смежен с двумя узлами, а не корневые внутренние узлы смежны ровно с тремя узлами. Дерево, состоящее из одного узла, является строго бинарным деревом по определению.

Два графа называются \emph{изоморфными}, если существует взаимно-однозначное соответствие между их вершинами, которое сохраняет смежность.

Пусть $k$~--- некоторое кардинальное число. Непустой граф $G$ называется \emph{полным $k$-дольным}, если его вершины можно разбить на $k$ непересекающихся подмножеств $X_1,...,X_k$ так, что нет рёбер, соединяющих вершины одного и того же подмножества $X_i$ и две любые вершины из разных $X_i, X_j$, $1\leqslant i, j\leqslant k$ смежны. В этом случае пишем $G=G[X_1,...,X_k]$.

Понятия теории графов, которым не дано явное определение, можно найти в ~\cite{BonMur(2008)}.

\section{Неравенство Гомори-Ху и полные двудольные графы}

\begin{definition}\label{def1.1}
Пусть $(X,d)$~--- метрическое пространство со спектром $\Sp(X)$ и пусть $r\in \Sp(X)$. Через $G_{r,X}$ будем обозначать граф, для которого $V(G_{r,X})=X$ и
$$
(\{u,v\}\in E(G_{r,X}))\Leftrightarrow (d(u,v)=r).
$$
\end{definition}

При $r=\diam X$ граф $G_{r,X}$ будем обозначать через $G_d$ и называть \emph{диаметральным графом} пространства $(X,d)$.

В этом разделе мы приводим простое доказательство неравенства Гомори-Ху~(\ref{eq1})  и характеризуем те пространства $X$, для которых достигается равенство, с помощью структурных свойств графов  $G_{r,X}$.

Для конечного множества $X$ через $|X|$ будем обозначать количество его элементов.
\begin{theorem}[~\cite{DDP(P-adic)} ]\label{th1.1}
Пусть $(X,d)$~--- конечное ультраметрическое про\-стран\-ство с $|X|\geqslant 2$.  Тогда $G_d=[X_1,...,X_k]$, $k\geqslant 2$.
 \end{theorem}

Докажем неравенство Гомори-Ху.

\begin{theorem}\label{th1.3}
Пусть $(X,d)$~--- конечное ультраметрическое пространство, $n=|X|$. Тогда выполнено неравенство
\begin{equation}\label{eq1.1}
|\Sp(X)|\leqslant n-1.
\end{equation}
\end{theorem}

\begin{proof}
Проведём доказательство по индукции. При $n=1$ неравенство~(\ref{eq1.1}) очевидно. Пусть~(\ref{eq1.1}) выполнено при всех $n\leqslant m$, $m\in \mathbb N$ и пусть $|X|=m+1$. Согласно теореме~\ref{th1.1} $G_d=[X_1,...,X_k]$.

По предположению индукции для  ультраметрических пространств $(X_1,d),...,(X_k,d)$, выполнены неравенства
$$
|\Sp(X_1)|\leqslant |X_1|-1, \dots, |\Sp(X_k)|\leqslant |X_k|-1.
$$
Отметим, что $$\Sp(X)=\Sp(X_1)\cup\dots\cup\Sp(X_k)\cup\{\diam(X)\}.$$
 Тогда имеем
 \begin{equation}\label{eq1.2}
|\Sp(X)|\leqslant |X_1|-1+\cdots +|X_k|-1 +1=m+1-k+1.
 \end{equation}
 Учитывая, что $k\geqslant 2$ мы получаем~(\ref{eq1.1}) при $n=m+1$.
\end{proof}

\begin{definition}
Обозначим через $\mathfrak{U}$ класс конечных  ультраметрических пространств $X$, спектр которых содержит ровно $|X|-1$ элементов.
\end{definition}

Анализируя доказательство теоремы~\ref{th1.3} получаем следующее

\begin{corollary}\label{cor1.1}
Пусть $(X,d)$~--- конечное ультраметрическое пространство, $|X|\geqslant 2$. Если  $(X,d)\in \mathfrak{U}$, то $G_d$~--- полный двудольный граф.
\end{corollary}

\begin{lemma}\label{lem7'}
Пусть $(X,d)\in \mathfrak U$, $G_d=[X_1,X_2]$ и пусть $(X_1,d)$, $(X_2,d)$~--- подпространства $X$ с носителями $X_1$ и $X_2$. Тогда
 \begin{equation}\label{eq1.0*}
\Sp(X)= \Sp(X_1)\cup \Sp(X_2)\cup \{\diam X\}
 \end{equation}
и
 \begin{equation}\label{eq1.2*}
\Sp(X_1)\cap\Sp(X_2)=\varnothing .
 \end{equation}
\end{lemma}

\begin{proof}
Равенство (\ref{eq1.0*}) следует из определения диаметрального графа и теоремы ~\ref{th1.1}. Это равенство влечёт неравенство
 \begin{equation}\label{eq1.1*}
|\Sp(X)|\leqslant|\Sp(X_1)|+|\Sp(X_2)|+1,
 \end{equation}
причём равенство в неравенстве достигается тогда и только тогда, когда выполнено~(\ref{eq1.2*}).
Так как $X\in \mathfrak{U}$, то $|\Sp(X)|=|X|-1$. Используя последнее равенство, неравенство Гомори-Ху и то, что
$$
X_1\cup X_2 = X,\quad X_1\cap X_2=\varnothing,
$$
находим
$$
|\Sp(X_1)|+|\Sp(X_2)|\leqslant |X_1|-1+|X_2|-1=|X|-2=|\Sp(X)|-1.
$$
Таким образом, выполнено неравенство обратное к~(\ref{eq1.1*}), а значит
\begin{equation}\label{eq1.3*}
|\Sp(X)|=|\Sp(X_1)|+|\Sp(X_2)|+1,
\end{equation}
что эквивалентно ~(\ref{eq1.2*}).
\end{proof}

Перейдём теперь к характеризации пространств $X\in \mathfrak{U}$ посредством графов $G_{r,X}$.

Пусть $G=(V,E)$~--- непустой граф, $V_0$~--- множество (возможно пустое) всех изолированных вершин графа $G$. Будем обозначать через $G'$ подграф графа $G$, порождённый множеством $V\backslash V_0$, т.е. $G'$~--- максимальный подграф графа $G$, не имеющий изолированных вершин.

\begin{lemma}\label{lem6*}
Пусть $G=(V,E)$~--- граф такой, что  $G'$ является полным двудольным и пусть $U\subseteq V$. Тогда, если индуцированный подграф $H=G(U)$ не пуст, то $H'$~--- полный двудольный граф.
\end{lemma}

Доказательство достаточно просто и мы оставляем его читателю.

\begin{theorem}\label{th7*}
Следующие утверждения равносильны для любого конечного ультраметрического пространства $(X,d)$ с $|X|\geqslant 2$.
\begin{itemize}
  \item [(i)] $(X,d) \in \mathfrak{U}$.
  \item [(ii)] Граф $G'_{r,X}$ является полным двудольным для любого $r\in \Sp(X)$.
\end{itemize}
\end{theorem}

\begin{proof}
Докажем  (i)$\Rightarrow$(ii) для любого ультраметрического $(X,d)$ с $2\leqslant |X|<\infty$. Доказательство проведём индукцией по $|X|$. Импликация  (i)$\Rightarrow$(ii), очевидно, имеет место при $|X|=2$. Предположим, что (i)$\Rightarrow$(ii) выполняется для всех конечных ультраметрических пространств $X$ с $2\leqslant |X|\leqslant n$. Пусть ультраметрическое пространство $(X,d)$ принадлежит $\mathfrak{U}$ и $|X|=n+1$. Докажем, что $G'_{r,X}$~--- полный двудольный граф для любого $r\in \Sp(X)$.

При $r=\diam X$ имеем $G_{r,X}=G_d$.  По следствию~\ref{cor1.1}, $G_{d}=[X_1,X_2]$. Следовательно  $G'_{r,X}=G'_{d}=G_d=G_{r,X}$ и всё доказано. Для доказательства того, что $G'_{r,X}$ является полным двудольным при $r<\diam X$, $r\in \Sp(X)$, воспользуемся леммой~\ref{lem7'}. Из~(\ref{eq1.0*}) и~(\ref{eq1.2*}) получаем, что утверждение
\begin{itemize}
  \item ($r\in \Sp(X_1)$ и $r\notin \Sp(X_2)$) или ($r\notin \Sp(X_1)$ и $r\in \Sp(X_2)$)
\end{itemize}
выполнено для любого $r\in \Sp(X)\backslash \{\diam X\}$. Не уменьшая общности, можно считать, что $r\in \Sp(X_1)$. Проверим принадлежность $X_1\in \mathfrak{U}$. Из ~(\ref{eq1.0*}) и~(\ref{eq1.2*}) следует равенство~(\ref{eq1.3*}).
Используя это равенство, равенство $|\Sp(X)|=|X|-1$ и неравенство Гомори-Ху для $X_1$ и $X_2$, получим
$$
|\Sp(X)|-1=|\Sp(X_1)|+|\Sp(X_2)|\leqslant |X_1|+|X_2|-2=|X|-2=|\Sp(X)|-1.
$$
Следовательно, $|\Sp(X_1)|+|\Sp(X_2)|=|X_1|+|X_2|-2$, что возможно только, если
$$
|\Sp(X_1)|=|X_1|-1 \, \text{ и } \, |\Sp(X_2)|=|X_2|-1.
$$
Таким образом, $(X_1,d)\in \mathfrak{U}$. Теперь так как $|X_1|<|X|$ и $r\in \Sp(X_1)$, то мы можем воспользоваться предположением индукции в соответствии с которым  $G'_{r,X_1}$~--- полный двудольный граф. Так как $r \notin \Sp(X_2)$ и $d(u,v)=\diam X>r$ при любых $u\in X_1$, $v\in X_2$, то $G'_{r,X_1}=G'_{r,X}$.  Таким образом, $G'_{r,X}$~--- полный двудольный граф для всех $r\in \Sp(X)$.

Импликация (i)$\Rightarrow$(ii) установлена.

Проверим (ii)$\Rightarrow$(i) для любого конечного ультраметрического пространства $(X,d)$ с $|X|\geqslant 2$.  Как и выше будем использовать индукцию по $|X|$. Для любого двухточечного $X$ мы, очевидно, имеем равенство в неравенстве Гомори-Ху. Следовательно, (ii)$\Rightarrow$(i) при $|X|=2$. Предположим, что (ii)$\Rightarrow$(i) выполнено при всех ультраметрических $X$ с $2\leqslant |X| \leqslant n$, $n$~--- фиксированное натуральное число. Пусть $(X,d)$~--- ультраметрическое пространство, $|X|=n+1$ и $G'_{r,X}$~--- полный двудольный граф для любого $r\in \Sp(X)$. Проверим, что $(X,d)\in \mathfrak{U}$. Пусть $G_d$~--- диаметральный граф пространства $(X,d)$. По теореме~\ref{th1.1} $G_d$~--- полный $k$-дольный. Следовательно $G'_d=G_d$ и в соответствии с предположением индукции $k=2$, т.е. $G_{d}=[X_1,X_2]$. Покажем, что $X_1, X_2 \in \mathfrak{U}$. Рассмотрим подпространство $X_1$. Если $|X_1|=1$, то всё очевидно. Для $|X_1|\geqslant 2$ спектр $\Sp(X_1)\neq \varnothing$. Если $r\in \Sp(X_1)$, то $r\in \Sp(X)$, а значит $G'_{r,X}$~--- полный двудольный граф. Пусть $G_{r,X}(X_1)$~--- подграф графа $G_{r,X}$, порождённый множеством $X_1$. Легко видеть, что $G_{r,X}(X_1)=G_{r,X_1}$. Так как $r\in \Sp(X_1)$, то $G_{r,X_1}$~--- не пустой граф. По лемме~\ref{lem6*} $G'_{r,X}(X_1)$~--- полный двудольный, т.е. $G'_{r,X_1}$~--- полный двудольный для любого $r\in \Sp (X_1)$. Так как $|X_1|<|X|$, то по предположению индукции $(X_1,d)\in \mathfrak{U}$. Аналогично доказывается, что $(X_2,d)\in \mathfrak{U}$. Проверим равенство~(\ref{eq1.2*}). Предположим, что
$$
\Sp(X_1)\cap\Sp(X_2)\neq\varnothing.
$$
Тогда для $r\in \Sp(X_1)\cap\Sp(X_2)$ множества $V(G'_{r,X})\cap X_1$ и $V(G'_{r,X})\cap X_2$ являются непустыми и образуют разбиение множества $V(G'_{r,X})$. Так как для любых $x_1\in X_1$ и $x_2\in X_2$ выполнено равенство $d(x_1,x_2)=\diam X$ и $\diam X>r$, то граф $G'_{r,X}$ является несвязным. Это противоречит (ii), так как любой полный двудольный граф связен. Равенство~(\ref{eq1.0*}), следующее из $G_d=G[X_1,X_2]$, вместе с~(\ref{eq1.2*})  даёт~(\ref{eq1.3*}). Выше было показано, что
$$
|\Sp(X_1)|=|X_1|-1 \, \text{ и } \, |\Sp(X_2)|=|X_2|-1.
$$
Подставляя эти равенства в ~(\ref{eq1.3*}) находим
$$
|\Sp(X)|=|X_1|-1+|X_2|-1+1=|X|-1,
$$
т.е. $X\in \mathfrak{U}$.

Таким образом, (ii)$\Rightarrow$(i) для любого конечного ультраметрического пространства $X$ с $|X|\geqslant 2$.
\end{proof}

Доказанная теорема и лемма~\ref{lem6*} дают
\begin{corollary}\label{cor8*}
Пусть $X\in \mathfrak{U}$ и $Y$~--- непустое подпространство $X$. Тогда $Y\in \mathfrak{U}$.
\end{corollary}

\section{Равенство в неравенстве Гомори-Ху и строго бинарные деревья}

\indent В этом разделе мы будем рассматривать только те ультраметрические пространства $(X,d)$, для которых $X\cap \Sp(X)=\varnothing$ (чего всегда можно добиться, переходя к пространству изометричному $X$).

Поставим каждому конечному ультраметрическому пространству $(X,d)$ в соответствие помеченное корневое $m$-арное дерево $T_X$ по следующему правилу. Если $X=\{x\}$~--- одноточечное множество, то $T_X$~--- дерево, состоящее из одного узла $\{x\}$, которое мы считаем строго бинарным по определению. Пусть $|X|\geqslant 2$. Корень дерева пометим меткой $\diam X$. Пусть $G_d$~--- диаметральный граф пространства $(X,d)$, $G_d=G[X_1,...,X_k]$. В этом случае будем считать, что дерево $T_X$ имеет $k$ узлов первого яруса с метками
\begin{equation}\label{eq1.6}
l_i:=
\begin{cases}
\diam X_i, &\text{если } |X_i|\geqslant 2\\
x,  &\text{если } X_i \text{~--- одноточечное множество} \\
&\text{с единственным элементом } x,
\end{cases}
\end{equation}
$i=1,...,k$. Узлы первого яруса, помеченные метками $x\in X$, будут листьями, а метками $\diam X_i$~--- внутренними узлами дерева $T_X$. Если на первом ярусе внутренних узлов нет, то дерево $T_X$ построено. В противном случае, повторяя описанную выше процедуру с $X_i\subset X$, соответствующими внутренним узлам первого яруса, получаем узлы второго яруса и т.д. Так как $|X|$ конечно, а величины $|Y|$, $Y\subseteq X$, строго убывают при движении вдоль любого пути, стартующего из корня, то на каком-то из ярусов все вершины будут листьями и построение $T_X$ завершается.

Построенное выше помеченное дерево $T_X$ будем называть \emph{представляющим деревом пространства} $(X,d)$. Отметим, что каждый элемент $x\in X$ приписан какому-то листу, а все внутренние узлы помечены метками $r\in \Sp(X)$. При этом разным листьям соответствуют разные $x\in X$, но различные внутренние узлы могут иметь совпадающие метки.
\begin{remark}
Представление конечного ультраметрического пространства $X$  с помощью помеченного корневого дерева $T_X$ известно. В~\cite{GurVyal(2012)} такое представление называется \emph{каноническим} и связывается со структурой позиционный игры, в которой игроки передвигаются из начальной позиции (корня дерева) к конечным (листьям дерева). Применение  диаметральных графов для построения $T_X$ выглядит новым.
\end{remark}
Корневое дерево, получающееся из $T_X$ путём ``стирания меток'', будем обозначать через $\overline{T}_X$. У неизометричныx пространств $X$ и $Y$ могут быть изоморфные деревья $\overline{T}_X$ и $\overline{T}_Y$.

\begin{example}\label{ex1}
На рис. 1 дерево $T_X$ является представляющим деревом  для пространства $(X,d)$, а дерево $T_Y$~--- представляющим деревом для пространства $(Y,\rho)$. Легко видеть, что пространства $(X,d)$ и $(Y,\rho)$ неизометричны, но $\Sp(X)=\Sp(Y)$ и $\overline{T}_X$ изоморфно $\overline{T}_Y$.
\begin{figure}[ht]\label{fig1}
\begin{center}
\includegraphics[width=0.8\linewidth]{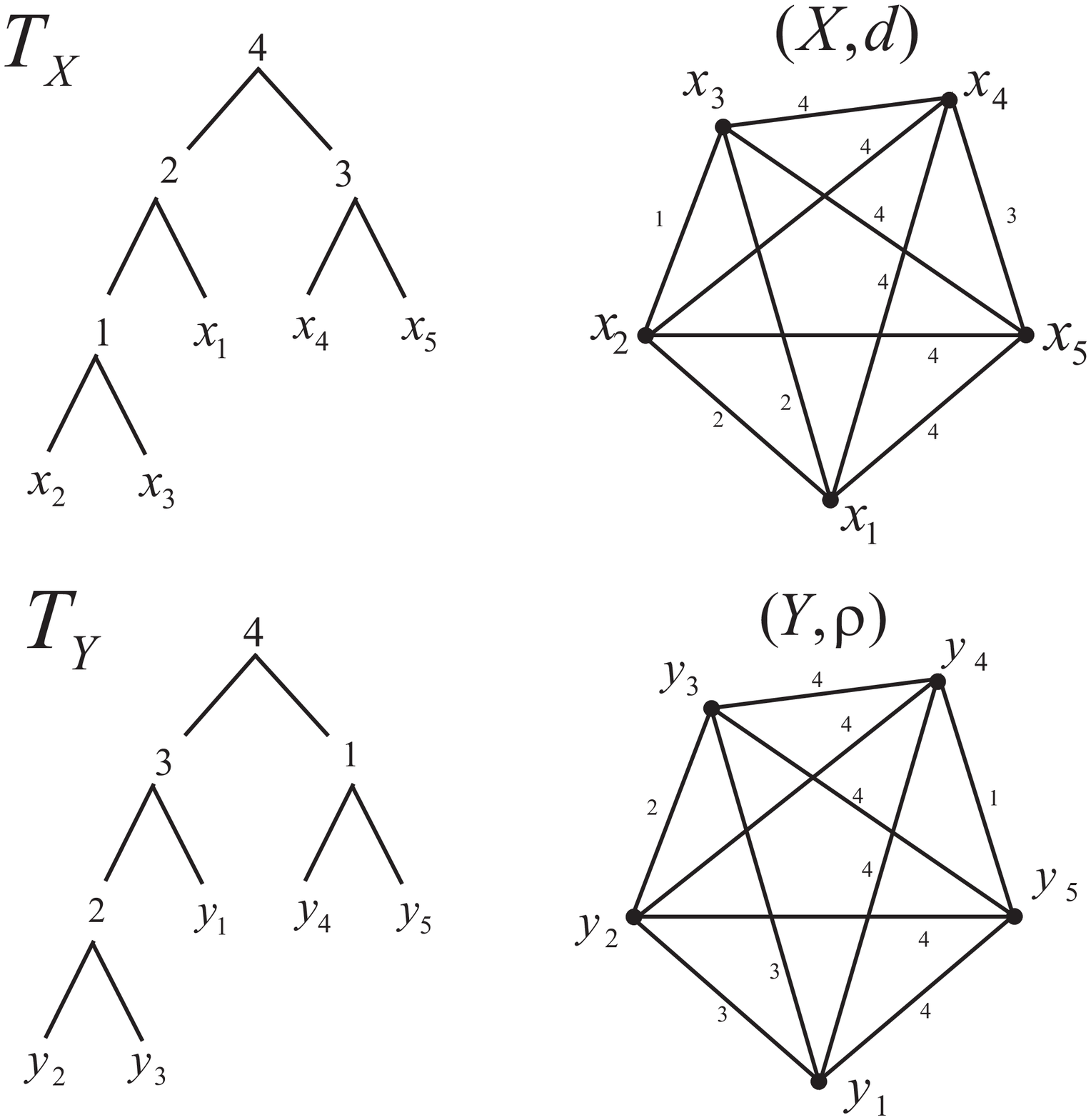}
\end{center}
\caption{Деревья неизометричных пространств могут быть изоморфны.}
\end{figure}
\end{example}

\begin{lemma}\label{lem1.2}
Конечное $m$-арное корневое дерево, имеющее $n$ листьев, $n\geqslant 2$, у которого количество потомков каждого внутреннего узла не менее двух, является строго бинарным тогда и только тогда, когда количество внутренних узлов равно $n-1$.
\end{lemma}

\begin{proof}
Покажем, что для любого конечного строго бинарного дерева с $n$  листьями количество внутренних узлов равно $n-1$. При $n=2$ имеем бинарный граф $T_2$, у которого один корень (внутренний узел) и два потомка (листа). Стартуя с $T_2$, любое строго бинарное дерево можно получить последовательным добавлением двух потомков одному из листьев. Такая процедура увеличивает на единицу как количество листьев, так и количество внутренних узлов. Разница между количество листьев и количеством узлов остаётся равной единице.

Пусть имеем конечное $m$-арное дерево с $n$ листьями и $n-1$ внутренним узлом, у которого количество потомков каждого внутреннего узла не менее двух. Покажем, что оно является строго бинарным. Пусть $T$~--- $m$-арное дерево, удовлетворяющее таким условиям, не является строго бинарным. Очевидно, любое конечное $m$-арное дерево можно получить последовательным добавлением не менее двух потомков одному из листьев, стартуя с одноточечного графа. Считаем, что в одноточечном графе корень является листом, а внутренние узлы отсутствуют.  Т.к. $T$~--- не строго бинарное дерево, то на каком-то этапе к некоторому узлу было добавлено не менее $k$ потомков, $k\geqslant 3$. Такая процедура увеличивает количество внутренних узлов на единицу и количество листьев на $k-1$. В этом случае разница между количеством листьев и количеством внутренних узлов изменилась в большую сторону. В  дальнейшем при добавлении ровно двух потомков такая разница не меняется, а при добавлении трёх и более может только увеличиться. Т.е. такое дерево $T$ не может иметь $n$ листьев и $n-1$  внутренних узлов.
\end{proof}

Сформулируем теперь новое необходимое и достаточное условие равенства в неравенстве Гомори-Ху.

\begin{theorem}\label{th2.1}
Пусть $(X,d)$~--- конечное ультраметрическое про\-стран\-ство с $|X|\geqslant 2$. $(X,d)\in \mathfrak{U}$ тогда и только тогда, когда представляющее  дерево $T_X$ пространства $(X,d)$ является строго бинарным и метки различных внутренних узлов  различны.
\end{theorem}

\begin{proof}
Пусть представляющее  дерево $T_X$ пространства $(X,d)$, $|X|=n$, является строго бинарным и метки, приписанные различным внутренним узлам, являются различными. Покажем, что $(X,d)\in \mathfrak{U}$. По построению, данное бинарное дерево имеет $n$ листьев.  Согласно лемме~\ref{lem1.2} внутренних узлов всего $n-1$.  Так как различные внутренние узлы имеют разные метки, а любая метка есть элемент множества $\Sp (X)$, то $|\Sp (X)|\geqslant n-1$. Неравенство, обратное к этому, есть неравенство Гомори-Ху. Что даёт $|\Sp(X)|=n-1$.

Обратно, пусть $(X,d)\in \mathfrak{U}$ и $|X|=n$. Тогда $G_d=G[X_1,X_2]$ и в соответствии с леммой~\ref{lem7'}
 \begin{equation}\label{eq8*}
\Sp (X_1)\cap \Sp (X_2) = \varnothing.
 \end{equation}
Равенство~(\ref{eq8*}) ведёт к тому что метки, различных внутренних узлов, различны. По индукции легко доказать, что  $T_X$ имеет $n-1$ внутренний узел и $n$ листьев. По лемме~\ref{lem1.2} $T_X$ является строго бинарным.
\end{proof}

\begin{remark}\label{rem14}
При построении представляющих деревьев $T_X$ для $(X,d)\in \mathfrak U$ каждое число из $\Sp(X)$ используется в качестве метки ровно один раз.
\end{remark}

\begin{corollary}\label{cor14*}
Пусть $(X,d)$~--- конечное ультраметрическое про\-стран\-ство с $|X|\geqslant 2$, для которого диаметральный граф $G_d$ является полным двудольным, $G_d=[X_1,X_2]$. Если $(X_i,d)\in \mathfrak{U}$, $i=1,2$, и $\Sp(X_1)\cap \Sp(X_2)=\varnothing$, то $(X,d)\in \mathfrak{U}$.
\end{corollary}

Следствие ~\ref{cor14*} может, до определённой степени, рассматриваться как обращение следствия ~\ref{cor8*}.

В 1948 году Оттером (Otter) ~\cite{Otter(1948)} было  подсчитано точное число $B_k$ не изоморфных конечных строго бинарных деревьев с $k+1$ листьями:
   \begin{equation}\label{eq3.1}
B_k=\begin{cases}
\frac{B_i(B_i+1)}{2}+\sum\limits_{j=0}^{i-1}B_{k-j-1}B_j, & \text{ если }\, \ k=2i+1,\\
\sum\limits_{j=0}^{i-1}B_{k-j-1}B_j, & \text{ если }\, \ k=2i.
\end{cases}
\end{equation}
Непосредственно можно найти начальные значения
\begin{equation}\label{eq9*}
B_1=B_2=1,\quad  B_3=2.
\end{equation}

Пусть $X$ и $Y$~--- конечные ультраметрические пространства. Будем писать $X\cong Y$, если корневые деревья $\overline{T}_X$ и $\overline{T}_Y$ изоморфны. В теореме~\ref{th16} будет показано, что число ультраметрических пространств $X$ из $\mathfrak U$, имеющих не изоморфные деревья $\overline{T}_X$, тоже может быть подсчитано по формуле~(\ref{eq3.1}).

Ниже мы отождествляем узлы дерева $T_X$ с приписанными им метками (см. замечание~\ref{rem14}).

\begin{lemma}\label{lem14}
Пусть $(X,d) \in \mathfrak{U}$ и пусть $x_1$, $x_2$~--- два различных листа дерева $T_X$. Тогда, если $(x_1,v_1,...,v_n,x_2)$~--- путь, соединяющий листья $x_1$ и $x_2$ в $T_X$, то
\begin{equation}\label{eq10}
d(x_1,x_2)=\max_{1\leqslant i\leqslant n} v_i.
\end{equation}
\end{lemma}

\begin{proof}
 Справедливость леммы легко установить индукцией по $|X|$. Формула~(\ref{eq10}) очевидно выполнена при $|X|=2$. Кроме того,~(\ref{eq10}) выполнена, если $\{x_1, x_2\}$~--- диаметральная пара пространства $(X,d)$, так как в этом случае путь $(x_1, v_1,...,v_n,x_2)$ проходит через корень дерева $T_X$, помеченный меткой $\diam X$. Если $\{x_1, x_2\}$ не является диаметральной парой, то $x_1$ и $x_2$ лежат в одной и той же компоненте дополнения диаметрального графа и можно применить предположение индукции (не сформулированное здесь явно).
\end{proof}

\begin{remark}\label{rem16} Если в формуле ~(\ref{eq10}) $\max\limits_{1\leqslant i \leqslant n} v_i =v_{i_0}$ и $v_{i_0}$~--- узел яруса $m_0$, а $v_{i_1}\neq v_{i_0}$~---внутренний узел дерева $T_X$, принадлежащий пути $(x_1, v_1,...,v_n, x_2)$ и лежащий на ярусе $m_1$, то $m_1>m_0$. В работе~\cite{GurVyal(2012)} это обстоятельство используется для \emph{построения} ультраметрического пространства, соответствующего заданному помеченному корневому дереву. Лемма~\ref{lem14} доказывает, что рассматриваемые нами представляющие деревья и деревья, реализующие каноническое представление в~\cite{GurVyal(2012)}, совпадают для пространств $(X,d)\in \mathfrak{U}$.
\end{remark}

Следующая лемма следует из теоремы 2 работы ~\cite{GurVyal(2012)} и теоремы~\ref{th2.1} настоящей работы.

\begin{lemma}\label{lem15}
Для любого конечного строго бинарного дерева $T$  найдётся ультраметрическое пространство $(X,d)\in \mathfrak{U}$ такое, что графы $T$ и $\overline{T}_X$ изоморфны.
\end{lemma}

Теперь мы можем использовать упомянутый выше результат Оттера для подсчёта числа пространств $(X,d)\in \mathfrak U$ с неизоморфными $\overline{T}_X$.

\begin{theorem}\label{th16}
Пусть $B_k$~--- число классов эквивалентности (относительно отношения $\cong$) ультраметрических пространств $(X,d)\in \mathfrak U$, для которых $|X|=k+1$. Тогда для чисел $B_k$ выполняется рекуррентное соотношение~(\ref{eq3.1}) с начальными значениями~(\ref{eq9*}).
\end{theorem}

\begin{proof}
Временно обозначим через $\tilde{B}_k$ число не изоморфных строго бинарных деревьев, имеющих $k+1$ листьев. В силу теоремы~\ref{th2.1}, $B_k\leqslant \tilde{B}_k$, где $B_k$~--- число всех неизоморфных строго бинарных деревьев, представляющих $(X,d)\in \mathfrak U$ с $|X|=k+1$. Лемма ~\ref{lem15} влечёт обратное неравенство $\tilde{B}_k \leqslant  B_k$. Следовательно $B_k =  \tilde{B}_k$  и мы можем использовать соотношение Оттера~(\ref{eq3.1}). Выполнение начальных условий~(\ref{eq9*}) проверяется непосредственно.
\end{proof}
Эквивалентность $X\cong Y$ может быть описана в чисто метрических терминах без использования представляющих деревьев $T_X$ и $T_Y$. Это будет сделано в теореме~\ref{th13*}.

Напомним, что в метрическом пространстве $(X,d)$ замкнутым шаром радиуса $r$ с центром в точке $t\in X$ называется множество
$$
B_r(t)=\{x\in X:d(x,t)\leqslant r\}.
$$
Для каждого $t\in X$ положим $\Sp_t(X):=\{d(x,t):x\in X, x\neq t\}$. Обозначим через $\textbf{B}_X$ множество всех шаров $B_r(t)$ с $r\in \Sp_t(X)$, т.е.
$$
\textbf{B}_X:=\{B_r(t):t\in X, r\in \Sp_t(X)\}.
$$
\begin{definition}\label{def12*}
Пусть $X$ и $Y$~--- метрические пространства. Будем говорить, что  отображение $F:X\to Y$ сохраняет шары, если для любых $Z\in \textbf{B}_X$ и $W\in \textbf{B}_Y$ выполнены соотношения
\begin{equation}\label{eq9}
F(Z)\in \textbf{B}_Y \,\text{ и } \, F^{-1}(W)\in \textbf{B}_X,
\end{equation}
где $F(Z)$~--- образ множества $Z$ при отображении $F$ и $F^{-1}(W)$~---  прообраз множества $W$ при этом отображении.
\end{definition}

Приведём пример отображения, сохраняющего шары.

\begin{example}
Пусть $(X,d)$~--- ультраметрическое пространство с $|X|\geqslant 2$ и $f:\Sp (X)\to (0,\infty)$~--- произвольная строго возрастающая функция. Положим
$$
d_f(x,y):=
\begin{cases}
0, &\text{если }x=y\\
f(d(x,y)), &\text{если } x\neq y.
\end{cases}
$$
Тогда $(X,d_f)$ тоже является ультраметрическим пространством и тождественное отображение $\mathrm{id}:(X,d)\to (X,d_f)$  сохраняет шары.
\end{example}

\begin{lemma}\label{lem20}
Пусть $(X,d)\in \mathfrak U$. Тогда для любого $r\in \Sp (X)$ множество $\textbf{B}_X$ содержит ровно один шар $B_r$ радиуса $r$.
\end{lemma}

\begin{proof}
При $|X|\leqslant 2$ утверждение очевидно. Предположим оно выполнено при $1\leqslant |X|\leqslant k$ с фиксированным $k\geqslant 2$. Рассмотрим $(X,d)\in \mathfrak U$ с $|X|=k+1$. Если $r=\diam X$, то единственный шар $B_r\in \textbf{B}_X$~--- это само пространство $X$. Пусть теперь $r\in \Sp(X)$, $r<\diam X$ и $B_r\in \textbf{B}_X$. Как и в лемме ~\ref{lem7'} рассмотрим подпространства $(X_1,d)$ и $(X_2,d)$ пространства $X$, такие, что $G_d=G[X_1,X_2]$. Используя~(\ref{eq1.0*}) и ~(\ref{eq1.2*}) мы, не уменьшая общности, можем предположить, что $r\in \Sp(X_1)\backslash \Sp(X_2)$. Пусть $x_1$~--- произвольная точка $X_1$, для которой $r\in \Sp_{x_1}(X)$. Положим
$$
B_r^1(x_1):=\{x\in X_1:d(x,x_1)\leqslant r\}.
$$
Для доказательства единственности шара $B_r$ достаточно установить равенство
\begin{equation}\label{eq12'}
B_r^1(x_1)=B_r.
\end{equation}
Сначала покажем, что
\begin{equation}\label{eq13'}
B_r\subseteq X_1.
\end{equation}
Действительно, если $B_r(t)\subseteq X_2\backslash X_1$, то $r\in \Sp (X_2)$, что противоречит принадлежности $r\in \Sp(X_1)\backslash \Sp(X_2)$. Предположим, что
$$
X_1\cap B_r(t)\neq \varnothing \neq X_2\cap B_r(t).
$$
Тогда существуют $y_1\in X_1 \cap B_r$ и $y_2\in X_2\cap B_r$. Так как в ультраметрическом пространстве диаметр шара не превосходит его радиуса, то $d(y_1,y_2)\leqslant r$. А так как $\{y_1,y_2\}$~--- диаметральная пара пространства $X$, то $d(y_1,y_2)=\diam X$. Следовательно $r\geqslant \diam X$, что противоречит условию $r<\diam X$ и доказывает ~(\ref{eq13'}). Включение ~(\ref{eq13'}) показывает, что $B^1_r(x_1)$, $B_r\in \textbf{B}_{X_1}$. Используя следствие~\ref{cor8*}, убеждаемся в том, что $X_1\in \mathfrak{U}$. Так как $|X_1|<|X|$, то $|X_1|\leqslant k$. По предположению индукции $\textbf{B}_{X_1}$ содержит единственный шар радиуса $r\in \Sp(X_1)$, что и доказывает равенство~(\ref{eq12'}).
\end{proof}

\begin{lemma}\label{lem21'}
Пусть $(X,d)\in \mathfrak U$ и пусть  $u$, $v$, $w$~--- различные внутренние узлы графа $T_X$ помеченные метками $r_u$, $r_v$, $r_{w}\in \Sp(X)$. Тогда узлы $v$ и $w$ являются потомками узла $u$ в том и только том случае, если
\begin{equation}\label{eq14*}
B_{r_u}=B_{r_v}\cup B_{r_w},
\end{equation}
где $B_{r_u}, B_{r_v}$ $B_{r_w}$~--- шары из $\textbf{B}_{X}$, имеющие радиусы $r_u$, $r_v$ и $r_w$ соответственно.
\end{lemma}

\begin{proof}
В процессе построения дерева $T_X$, описанном в начале настоящего раздела, носитель $X$ пространства $(X,d)\in \mathfrak U$ последовательно разбивается на пары непересекающихся подмножеств до тех пор пока в результате не получаются одноточечные множества. Например, при $G_d=G[X_1,X_2]$ множество  $X$ разбивается на $X_1$ и $X_2$. В свою очередь $X_1$, при $\diam X_1>0$, разбивается на $X_{1,1}$  и $X_{1,2}$; $X_2$, при $\diam X_2>0$ на $X_{2,1}$ и $X_{2,2}$  и т.д. Причём в качестве меток внутренних узлов дерева $T_X$ выступают числа $\diam X$, $\diam X_1$, $\diam X_2$, $\diam X_{1,1}$,... (если они не являются нулями).

Для доказательства настоящей леммы достаточно установить, что любое подмножество ненулевого диаметра из последовательности  $X$,  $X_1$,  $X_2$,  $X_{1,1}$, $\dots$ принадлежит $\textbf{B}_X$. Пусть $\diam X_i>0$, $i=1,2$. Проверим, что \begin{equation}\label{eq18t}
X\in \textbf{B}_X,\quad X_1\in \textbf{B}_X\, \text{ и }\,  X_2\in \textbf{B}_X.
\end{equation}
Включение $X\in \textbf{B}_X$ выше. Пусть $\{x_1, x_2\}$~--- диаметральная пара пространства $X_1$ и $r:=\diam X_1$. Проверим, что
\begin{equation}\label{eq15s}
B_{r}(x_1)=X_1.
\end{equation}
Действительно, если $x_3$~--- произвольная точка $X_1$, то $\diam X_1\geqslant d(x_1,x_3)$, т.е. $x_3\in B_r(x_1)$. Следовательно
\begin{equation}\label{eq16s}
X_1\subseteq B_{r}(x_1).
\end{equation}
Пусть $x\in X\backslash X_1$. Тогда $x\in X_2$ и, по определению диаметрального графа,
$$
d(x,x_1)=\diam X.
$$
Из неравенства $\diam X>\diam X_1$ следует, что $x\notin B_r(x_1)$. Таким образом, $X_1\supseteq B_r(x_1)$, что вместе с~(\ref{eq16s}) даёт~(\ref{eq15s}). Так как $B_r(x_1)\in \textbf{B}_X$, то  $X_1\in \textbf{B}_X$. Аналогично проверяется, что $X_2\in \textbf{B}_X$.

Простая модификация рассуждений, проведённых при доказательстве равенства~(\ref{eq15s}) показывает, что импликация
\begin{equation}\label{eq19t}
(Y\in \textbf{B}_X, \quad Z\in \textbf{B}_Y)\Rightarrow (Z\in \textbf{B}_X).
\end{equation}
справедлива для произвольных $Y, Z\subseteq X$. Теперь доказательство того, что все элементы последовательности $X$, $X_1$, $X_2$, $X_{1,1}$ $\dots$ принадлежат $\textbf{B}_X$ легко завершается. Например, заменяя в~(\ref{eq18t}) $X$ на $X_1$, a $X_1$ на $X_{1,1}$ получим $X_{1,1}\in \textbf{B}_{X_1}$. Таким образом, посылка импликации~(\ref{eq19t}) выполнена с $Y=X$ и $Z=X_{1,1}$. Следовательно $X_{1,1}\in \textbf{B}_X$ и т.д..
\end{proof}

\begin{theorem}\label{th13*}
Пусть $(X,d)$, $(Y,\rho) \in \mathfrak{U}$. Эквивалентность $X\cong Y$ выполняется тогда и только тогда, когда существует биективное отображение $F:X\to Y$, сохраняющее шары.
\end{theorem}
\begin{proof} Предположим, что $X\cong Y$. Пусть $V_X = V(\overline{T}_X)$ и $V_Y = V(\overline{T}_Y)$~--- множества вершин графов $\overline{T}_X$ и $\overline{T}_Y$ соответственно. Эквивалентность $X\cong Y$ означает, что существует биективное отображение $\Psi:V_X\to V_Y$, сохраняющее отношение смежности между вершинами графов $\overline{T}_X$ и $\overline{T}_Y $. Множество $V(\overline{T}_X)$ ($V(\overline{T}_Y)$) находится в естественном взаимно однозначном соответствии с множеством меток, которыми были помечены вершины графа $T_X (T_Y)$ (см. замечание~\ref{rem14}). Для простоты будем отождествлять эти множества.  Биекция $\Psi$ отображает множество листьев графа $\overline{T}_X$ (= множество $X$) на множество листьев графа $\overline{T}_Y$ (= множество $Y$), так как листья ~--- это в точности вершины степени 1. Обозначим через $\Phi$ сужение $\Psi$  на $X$, $\Phi=\Psi|_{X}$, и покажем, что биективное отображение $\Phi:X\to Y$, рассматриваемое как отображение между метрическими пространствами $(X,d)$ и $(Y,\rho)$, сохраняет шары. Доказательство проведём индукцией по $|X|$. Сохранение шаров очевидно при $|X|\leqslant 2$. Предположим, что $\Phi$ сохраняет шары если $|X|\leqslant n$, $n$~--- фиксированное натуральное число $\geqslant 2$. Докажем, что шары сохраняются и при $|X|=n+1$.

Положим $r_1:=\diam X$ и $q_1:=\diam Y$. В пространстве $(X,d)$ существует единственный шар, принадлежащий $\textbf{B}_X$ и имеющий радиус $r_1$~--- это само множество $X$. Аналогично $Y$~--- единственный шар из $\textbf{B}_Y$, имеющий радиус $q_1$. Так как
$$
\Phi(X)=Y \,\text{ и } \, \Phi^{-1}(Y)=X,
$$
то всё доказано. Пусть теперь $B_r\in \textbf{B}_X$, $r\in \Sp(X)$ и $r<\diam X$. Для того, чтобы показать справедливость принадлежности $\Phi (B_r)\in \textbf{B}_Y$ будем использовать предположение индукции. Заметим, что $\Psi(r_1)=q_1$ так как $r_1$~--- единственная вершина степени $2$ в $\overline{T}_X$, а $q_1$~--- единственная вершина степени $2$  в $\overline{T}_Y$. Пусть $r_{11}$ и $r_{12}$~--- две вершины первого яруса в $\overline{T}_X$, a $q_{11}$ и $q_{12}$ две вершины первого яруса в $\overline{T}_Y$. Так как $\overline{T}_X$ и $\overline{T}_Y$ строго бинарные деревья, то других вершин первого яруса в $\overline{T}_X$ и $\overline{T}_Y$ нет. Отображение $\Psi$ переводит смежные вершины в смежные, следовательно, не уменьшая общности, можно считать, что $\Psi(r_{11})=q_{11}$ и $\Psi(r_{12})=q_{12}$. Если удалить из $\overline{T}_X$ вершину $r_1$, a из $\overline{T}_Y$ вершину $q_1$, то $\overline{T}_X$ распадается на два строго бинарных дерева $\overline{T}^{1}_X $ и $\overline{T}^{2}_X$. Аналогично, удаляя из $\overline{T}_Y$ вершину $q_1$ получим строго бинарные деревья $\overline{T}^{1}_Y $ и $\overline{T}^{2}_Y$. Деревья $\overline{T}^{1}_X$ и $\overline{T}^{2}_X$ не имеют общих узлов и мы можем предполагать, что, выбранный ранее узел $r$, принадлежит $V(\overline{T}^{1}_X)$ и $\Psi(V(\overline{T}^{1}_X))=V(\overline{T}^{1}_Y)$. Пусть $X_1$, $X_2$, $Y_1$, $Y_2$~--- множества листьев деревьев $\overline{T}^{1}_X$, $\overline{T}^{2}_X$, $\overline{T}^{1}_Y$ и $\overline{T}^{2}_Y$ соответственно. Заметим, что
\begin{equation}\label{eq12}
|X_1|\leqslant n.
\end{equation}
Сужение $\Psi|_{V(\overline{T}^{1}_X)}$ является изоморфизмом строго бинарных деревьев $\overline{T}^{1}_X$ и $\overline{T}^{1}_Y$. Рассмотрим теперь подпространство $(X_1,d)$ пространства $(X,d)$ и подпространство $(Y_1,\rho)$ пространства $(Y,\rho)$. Отождествляя вершины графов $T_{X_1}$ и $T_{X_2}$ с их метками, видим, что граф $\overline{T}_{X_1}$ совпадает с графом $\overline{T}^{1}_X$, а граф $\overline{T}_{Y_1}$ с $\overline{T}^{1}_Y$. Следовательно $\Psi|_{V(\overline{T}^{1}_X)}$~--- изоморфизм графов $\overline{T}_{X_1}$ и $\overline{T}_{Y_1}$. Неравенство~(\ref{eq12}) и предположение индукции показывают, что функция $\Psi|_{X_1}$ сохраняет шары как отображение из $(X_1,d)$ в $(Y_1,\rho)$. Пусть $B_r(x_1)\in \textbf{B}_{X}$, где $r$~--- число из $\Sp(X)\backslash \{r_1\}$ выбранное выше. Так как $\Sp(X_1)\cap \Sp(X_2)=\varnothing$ и $r\in \Sp(X_1)$, то $x_1\in X_1$. Если $x_2$~--- произвольная точка из $X_2$, то $d(x_1,x_2)=\diam X=r_1>r$. Следовательно
\begin{equation}\label{eq13}
B_r(x_1)\subseteq X_1,
\end{equation}
a значит $B_r(x_1)\in \textbf{B}_{X_1}$. Следовательно $\Psi|_{X_1}(B_r(x_1))$~--- шар в $Y_1$. Включение~(\ref{eq13}) даёт равенство
$$
\Psi|_{X_1}(B_r(x_1))=\Phi(B_r(x_1)).
$$
Кроме того, непосредственно из определений графов $\overline{T}^{1}_{X}$, $\overline{T}^{2}_{X}$ и $\overline{T}^{1}_{Y}$, $\overline{T}^{2}_{Y}$ следует, что
$$
\Phi(X_1)=\Psi(X_1)=Y_1,\quad \Phi(X_2)=\Psi(X_2)=Y_2, \quad Y_1\cap Y_2=\varnothing.
$$
Таким образом, $\Phi(B_r(x_1))\subseteq Y_1$, а значит $\Phi(B_r(x))$~--- шар в $Y$. Аналогично доказывается, что $\Phi^{-1}(B_q(y))\in \textbf{B}_X$ для $y\in Y$ и $q\in \Sp_y(X)\backslash\{\diam Y\}$. Таким образом, отображение $\Phi:X\to Y$ сохраняет шары.

Пусть теперь $(X,d)$, $(Y,\rho)\in \mathfrak{U}$ и $\Phi:X\to Y$ сохраняющая шары биекция. Нужно доказать, что $X\cong Y$. Для этого продолжим отображение $\Phi:X\to Y$ до отображения $\Phi_*:V(T_X)\to V(T_Y)$, где $V(T_X)$ и $V(T_Y)$~--- множества вершин строго бинарных деревьев, представляющих пространства $X$ и $Y$ соответственно, и покажем, что отображение $\Phi_*$, сохраняет отношение смежности. На множестве листьев дерева $T_X$ отображение $\Phi_*$ считаем совпадающим с $\Phi$. Для определения $\Phi_*$ на множестве внутренних узлов заметим, что отображение
\begin{equation}\label{eq15*}
\textbf{B}_X\ni Z\mapsto \Phi(Z)\in \textbf{B}_Y
\end{equation}
является биекцией множества $\textbf{B}_X$ на множество $\textbf{B}_Y$. (Последнее легко вытекает из того, что $F$ биективно и сохраняет шары). По лемме~\ref{lem20} отображение~(\ref{eq15*}) порождает биективное отображение между $\Sp(X)$ и $\Sp(Y)$
\begin{equation}\label{eq16*}
\Sp(X)\ni r \mapsto B_r \mapsto \Phi(B_r)\mapsto \diam \Phi(B_r)\in \Sp(Y),
\end{equation}
где $B_r$~--- единственный шар, принадлежащий $\textbf{B}_X$ и имеющий радиус (= диаметр) $r$. Так как метки $r$, соответствующие различным внутренним узлам дерева $T_X$, различны и различны метки внутренних узлов дерева $T_Y$, то ~(\ref{eq16*}) порождает естественную биекцию между внутренними узлами этих деревьев, которая и осуществляет продолжение $\Phi$ до $\Phi_*$. (Таким образом, если $v$~--- внутренний узел дерева $T_X$, имеющий метку $r\in \Sp(X)$, то $\Phi_*(v)$~--- внутренний узел дерева $T_Y$, помеченный меткой $\diam \Phi(B_r)$, которую задаём с помощью~(\ref{eq16*})).

Для завершения доказательства осталось проверить, что отображение $\Phi_*:V(T_X)\to V(T_Y)$ сохраняет отношение смежности. Пусть $v_1^*$ и $v_2^*$~--- смежные вершины графа $T_X$. Не уменьшая общности, можно считать, что узел $v_2^*$ является потомком узла $v_1^*$. Тогда $v_1^*$~--- внутренний узел. Так как $T_X$~--- строго бинарное дерево, то любой внутренний узел имеет ровно два потомка. Пусть $v_3^*$~--- потомок $v_1^*$ отличный от $v_2^*$. Возможны следующие три случая:
\begin{itemize}
  \item [(i)] узлы $v_2^*$ и $v_3^*$~--- внутренние;
  \item [(ii)] один из узлов $v_2^*$, $v_3^*$ является внутренним, а другой узел~--- лист;
  \item [(iii)] оба узла $v_2^*$ и $v_3^*$ являются листьями.
\end{itemize}
Допустим, что имеет место случай (i). Докажем, что двухэлементные множества $\{\Phi_*(v_1^*),\Phi_*(v_2^*)\}$ и $\{\Phi_*(v_1^*),\Phi_*(v_3^*)\}$~--- ребра графа $T_Y$. Пусть $r_1^*$, $r_2^*$, $r_3^*$~--- метки узлов $v_1^*$, $v_2^*$, $v_3^*$. Тогда по лемме~\ref{lem21'} имеем
\begin{equation}\label{eq17*}
B_{r_1^*}=B_{r_2^*}\cup B_{r_3^*},
\end{equation}
где $B_{r_i^*}\in \textbf{B}_X$ шар радиуса $r_i^*$, $i=1, 2, 3$. Равенство~(\ref{eq17*}) влечёт равенство
\begin{equation}\label{eq18*}
\Phi(B_{r_1^*})=\Phi(B_{r_2^*})\cup \Phi(B_{r_3^*}).
\end{equation}
Узлы $v_i^y := \Phi_*(v_i^*)$ помечены в $T_Y$ метками $\diam \Phi(B_{r_i^*})$, $i=1,2,3$. В силу леммы~\ref{lem21'} равенство~(\ref{eq18*}) влечёт смежность $v_1^y$ и $v_2^y$ и смежность $v_1^y$ и $v_3^y$.

Аналогично можно доказать смежность $\Phi_*(v_1^*)$ и $\Phi_*(v_2^*)$, если $v_2^*$~--- лист или листом является $v_3^*$. Мы не будем делать это подробно отметим только, что, если, например, $w$~--- лист, а $v$~--- внутренний узел, то ~(\ref{eq14*}) остаётся верным после замены $B_{r_w}$ на одноточечное множество, единственная точка которого это тот элемент из $X$, который использован в качестве метки для $w$.

Аналогично проверяем, что, если $u_1^*$ и $u_2^*$~--- смежные вершины графа $T_Y$, то $\Phi_*^{-1}(u_1^*)$ и $\Phi_*^{-1}(u_2^*)$~--- смежные вершины в $T_X$.

Таким образом, если существует сохраняющая шары биекция $\Phi:X\to Y$, то можно построить изоморфизм $\Phi_*:V(\overline{T}_X)\to V(\overline{T}_Y)$, т.е. $X\cong Y$.
\end{proof}

\begin{remark}\label{rem25}
Теорема~\ref{th13*} остаётся справедливой, если вместо пространств $X, Y\in \mathfrak U$ рассматривать произвольные конечные ультраметрические пространства.
\end{remark}

В заключение приведём ещё один результат близкий к неравенству Гомори-Ху
\begin{theorem}\label{th26*}
Пусть $(X,d)$~--- конечное ультраметрическое пространство. Тогда
\begin{equation}\label{eq27p}
|\textbf{B}_X|\leqslant |X|-1
\end{equation}
причём равенство в этом неравенстве имеет место тогда и только тогда, когда представляющее дерево $T_X$ пространства $X$ является строго бинарным.
\end{theorem}

Неравенство Гомори-Ху $|\Sp(X)|\leqslant |X|-1$ легко вывести из~(\ref{eq27p}).  Достаточно заметить, что шары из $\textbf{B}_X$, имеющие различные радиусы, не могут совпадать, а значит $|\Sp(X)|\leqslant |\textbf{B}_X|$.

\section{Изометрические типы пространств из $\mathfrak{U}$ в пространстве Громова-Хаусдорфа}

Для любого метрического пространства $(X,d)$ через $\is (X)$ будем обозначать изометрический тип этого пространства, т.е. $\is (X)$~--- это класс всех метрических пространств изометричных $(X,d)$. В этом разделе мы установим плотность множества $\{\is(X):(X,d)\in \mathfrak U\}$ в пространстве Громова-Хаусдорфа изометрических типов компактных ультраметрических пространств  и покажем, что в пространстве Громова-Хаусдорфа, порождённом конечными ультраметрическими пространствами $(X,d)$ с $|X|\leqslant n$ множество $\{\is(X):(X,d)\in \mathfrak U, |X|=n\}$ является открытым и всюду плотным. Таким образом любые конечные ультраметрические пространства ``малыми возмущениями'' переводятся в пространства из $\mathfrak U$, которые являются ``устойчивыми'' при таких возмущениях.

Для удобства напомним определение расстояния по Громову-Хаусдорфу (см., например, ~\cite[стр. 254]{Burago(2001)}).

\begin{definition}\label{def4.1}
Пусть $(X,d_X)$ и $(Y,d_Y)$~--- ограниченные метрические пространства  и пусть $\varepsilon>0$. Расстояние Громова-Хаусдорфа $d_{GH}(\is(X),\is(Y))$ меньше чем $\varepsilon$ тогда и только тогда, когда существует метрическое пространство $(Z,d_Z)$ c подпространствами $X'$ и $Y'$, для которых $\is(X')=\is(X)$, $\is(Y')=\is(Y)$ и
\begin{equation}\label{eq4.1}
X'\subseteq \bigcup\limits_{y\in Y'}O_{\varepsilon}(y)\, \text{ и } \, Y'\subseteq \bigcup\limits_{x\in X'}O_{\varepsilon}(x),
\end{equation}
где $O_{\varepsilon}(t)=\{z\in Z:d_Z(t,z)<\varepsilon \}$~--- открытый шар из $(Z,d_Z)$ с центром в точке $t$ и радиуса $\varepsilon$.
\end{definition}

Величина $d_{GH}$ удовлетворяет неравенству треугольника~\cite[стр. 255]{Burago(2001)},  неотрицательна и если $(X,d_X)$ и $(Y,d_Y)$ компактны, то
$$
(d_{GH}(\is(X),\is(Y))=0)\Leftrightarrow (\is X = \is Y).
$$
Таким образом, $d_{GH}$~--- метрика на множестве изометрических типов компактных метрических пространств (~\cite[стр. 259, теорема 7.3.30]{Burago(2001)}).

\begin{theorem}\label{th4.2}
Множество $\{\is(X):X\in \mathfrak U\}$ является плотным в пространстве изометрических типов компактных ультраметрических пространств, наделённом метрикой Громова-Хаусдорфа $d_{GH}$.
\end{theorem}

Для доказательства теоремы~\ref{th4.2} заметим, что если подмножество $Y$ метрического пространства $(X,d)$ является $\varepsilon$-сетью (в $X$), то $d_{GH}(\is(X),\is(Y))\leqslant \varepsilon$. Так как для компактного $(X,d)$ конечная $\varepsilon$-сеть существует при любом $\varepsilon>0$, то достаточно установить следующее

\begin{statement}\label{st27}
Пусть $(X,d)$~--- конечное ультраметрическое пространство. Тогда для любого $\varepsilon>0$ найдётся $(Y,\rho)\in \mathfrak U$ такое, что $|X|=|Y|$ и
\begin{equation}\label{eq4.2}
d_{GH}(\is(X),\is(Y))<\varepsilon.
\end{equation}
\end{statement}

\begin{proof}
Занумеруем точки множества $X$ в последовательность $x_1,...,x_n$, $n=|X|$. Пусть $\mathbb{R}_{\infty}^{n}$~--- линейное пространство векторов $(t_1,...,t_n)$ с нормой $\max\limits_{1\leqslant i\leqslant n}|t_i|$. Каждой точке $z\in X$ поставим в соответствие функцию $\delta_z:X\to \mathbb R$, определённую равенством $\delta_z(x_n)=d(x_n,z)$, $x_n\in X$. Мы будем рассматривать эту функцию как вектор из $\mathbb R_{\infty}^n$ с координатами $d(x_1,z)$,...,$d(x_n,z)$. Отображение  $X\ni z\mapsto \delta_z \in \mathbb R_{\infty}^n$ является изометрическим вложением пространства $(X,d)$ в $\mathbb R_{\infty}^n$ (см., например, \cite[стр. 14, теорема 1.5.1]{Searcoid(2007)}). Этот факт и приведённое выше определение~\ref{def4.1} показывают, что для доказательства утверждения  достаточно найти ультраметрику $\rho_{\varepsilon}$ на $X$ так, что $(X,\rho_{\varepsilon})\in \mathfrak U$ и
\begin{equation}\label{eq4.3}
|d(x,y)-\rho_{\varepsilon}(x,y)|<\varepsilon
\end{equation}
при всех $x,y\in X$.

Будем строить искомую ультраметрику $\rho_{\varepsilon}$ индукцией по $|X|$. Такое построение легко осуществить при $|X|\leqslant 2$. Пусть $n\geqslant 2$~--- фиксированное натуральное число. Предположим, что для любого $\varepsilon >0$  и любого ультраметрического пространства $(X,d)$  с $|X|\leqslant n$ найдётся ультраметрика $\rho_{\varepsilon}$ на $X$, для которой $(X,\rho_{\varepsilon})\in \mathfrak U$ и ~(\ref{eq4.3}) выполнено для всех $x,y \in X$. Рассмотрим произвольное ультраметрическое пространство $(X,d)$ с $|X|=n+1$. По теореме~\ref{th1.1} диаметральный граф пространства $(X,d)$ является полным $k$-дольным с $k\geqslant 2$. Возможны следующие два случая: (i) $k=2$, (ii) $k\geqslant 3$.

Вначале проведём построение $\rho_{\varepsilon}$ в случае (i). Пусть $G_d=G[X_1,X_2]$. Так как $|X_i|\leqslant n$, то по предположению индукции для любого $\varepsilon>0$  на $X_i$ существует ультраметрика $\rho_{\varepsilon}^i$ такая, что $(X_i,\rho_{\varepsilon}^i)\in \mathfrak U$ и
\begin{equation}\label{eq4.4}
|d(x,y)-\rho_{\varepsilon}^i(x,y)|<\frac{\varepsilon}{4}
\end{equation}
для всех $x,y\in X_i$ и $i=1, 2$. Выберем $\varepsilon $  в~(\ref{eq4.4}) так, чтобы дополнительно выполнялось неравенство
\begin{equation}\label{eq4.5}
\frac{\varepsilon}{2}<\min\limits_{\substack{r,t \in\Sp(X) \\ r\neq t}}|r-t| \wedge \min\limits_{r\in \Sp(X)} r.
\end{equation}
Предположим теперь, что на $X_1$ можно ввести новую ультраметрику $\bar{\rho}^1_{\varepsilon}$ так, что для всех $x,y\in X_1$
\begin{equation}\label{eq4.6}
|\rho_{\varepsilon}^1(x,y)-\bar{\rho}_{\varepsilon}^1(x,y)|<\frac{\varepsilon}{4}
\end{equation}
 и
\begin{equation}\label{eq4.7}
\Sp(X_1,\bar{\rho}^1_{\varepsilon})\cap\Sp(X_2,\rho^2_{\varepsilon})=\varnothing
\end{equation}
и
\begin{equation}\label{eq4.8}
(X_1,\bar{\rho}_{\varepsilon}^1)\in \mathfrak U.
\end{equation}
Определим функцию $\rho_{\varepsilon}:X\times X\to \mathbb R^+$ по правилу
\begin{equation}\label{eq4.9}
\rho_{\varepsilon}(x,y):=
\begin{cases}
\diam X, &\text{если }\, x\in X_i, y\in X_j, i\neq j,\\
\bar{\rho}^1_{\varepsilon}, &\text{если }\, x,y\in X_1,\\
\rho^2_{\varepsilon}, &\text{если }\, x,y\in X_2,
\end{cases}
\end{equation}
где $\diam X=\max\limits_{x,y\in X}d(x,y)$. Покажем, что $\rho_{\varepsilon}$~--- ультраметрика на $X$. В проверке нуждается только сильное неравенство треугольника, остальные характеристические свойства очевидны. Пусть $x,y,z\in X$ докажем, что
\begin{equation}\label{eq4.10}
\rho_{\varepsilon}(x,y)\leqslant \max\{\rho_{\varepsilon}(x,z),\rho_{\varepsilon}(y,z)\}.
\end{equation}
Достаточно рассмотреть~(\ref{eq4.10}) при $x,y\in X_1$ и $z\in X_2$ или при $x,y\in X_2$ и $z\in X_1$, т.к. в остальных случаях~(\ref{eq4.10}) следует непосредственно из~(\ref{eq4.9}). Пусть $x,y\in X_1$ и $z\in X_2$. Используя~(\ref{eq4.9}) имеем
\begin{equation}\label{eq4.10'}
\max(\rho_{\varepsilon}(x,z),\rho_{\varepsilon}(y,z))=\diam X,
\end{equation}
а из~(\ref{eq4.9}), ~(\ref{eq4.6})  и ~(\ref{eq4.4}) следует, что
\begin{equation}\label{eq4.11}
\rho_{\varepsilon}(x,y)=\bar{\rho}_{\varepsilon}^1(x,y)\leqslant \frac{\varepsilon}{4}+\rho_{\varepsilon}^1(x,y)<\frac{\varepsilon}{2}+d(x,y).
\end{equation}
Так как  $x,y \in X_1$, то $\{x, y\}$ не является диаметральной парой для $(X,d)$ и, используя ~(\ref{eq4.5}), получим
$$
\diam X>d(x,y)+\frac{\varepsilon}{2}.
$$
Таким образом,
\begin{equation}\label{eq4.12}
\rho_{\varepsilon}(x,y)<\diam X
\end{equation}
для всех $x,y\in X_1$. Неравенство ~(\ref{eq4.12}) и равенство~(\ref{eq4.10'}) доказывают~(\ref{eq4.10}). Аналогично можно установить~(\ref{eq4.12}) при $x,y\in X_2$, что доказывает~(\ref{eq4.10}) при $z\in X_1$ и $x,y\in X_2$. Таким образом, $\rho_{\varepsilon}$~--- ультраметрика на $X$. То что ~(\ref{eq4.12}) имеет место при всех $x,y\in X_i$, $i=1,2$ приводит к совпадению диаметральных графов пространств $(X,d)$ и $(X,\rho_{\varepsilon})$. Следовательно, $G_{\rho_{\varepsilon}}$~--- полный двудольный. Из этого свойства,~(\ref{eq4.7}), ~(\ref{eq4.8}) и того, что $(X_2,\rho_{\varepsilon}^2)\in \mathfrak U$ по следствию~\ref{cor14*} вытекает принадлежность $(X,\rho_{\varepsilon})\in \mathfrak U$. Заметим теперь, что ~(\ref{eq4.9}), ~(\ref{eq4.12}),~(\ref{eq4.6}) и ~(\ref{eq4.4}) приводит к неравенству ~(\ref{eq4.3}) при всех $x, y\in X$. Таким образом, $\rho_{\varepsilon}$, определённая формулой~(\ref{eq4.9}), обладает желаемыми свойствами, если существует $\bar{\rho}^{1}_{\varepsilon}$, для которой выполняются соотношения ~(\ref{eq4.6})~-- (\ref{eq4.8}).

Покажем как построить ультраметрику $\bar{\rho}^1_{\varepsilon}$, удовлетворяющую~(\ref{eq4.6})~-- (\ref{eq4.8}). Если
\begin{equation}\label{eq4.13}
\Sp(X_1,\rho^1_{\varepsilon})\cap \Sp(X_2,\rho^2_{\varepsilon})=\varnothing,
\end{equation}
то достаточно взять $\bar{\rho}_{\varepsilon}^1=\rho_{\varepsilon}^1$. Пусть~(\ref{eq4.13}) не выполнено. Выберем
\begin{equation}\label{eq4.14}
\Delta = \min\limits_{\substack{r,t\in \Sp^1\\ r\neq t}} |r-t|\wedge \min\limits_{r \in \Sp^1}r \wedge\frac{\varepsilon}{4},
\end{equation}
где $\Sp^1=\Sp(X_1,\rho_{\varepsilon}^1)$. При таком выборе
\begin{equation}\label{eq4.15}
(r-\Delta, r)\cap (t-\Delta,t)=\varnothing, \quad r \geqslant \Delta > 0
\end{equation}
для любых различных $r,t\in \Sp^1$. Пусть $f:\Sp^1\to (0,\infty)$~--- произвольная функция, для которой
\begin{equation}\label{eq4.16}
f(r)\in (r-\Delta,r)\backslash \Sp(X_2,\rho^2_{\varepsilon}).
\end{equation}
Из~(\ref{eq4.15}) следует, что
\begin{equation}\label{eq4.17}
(r<t)\Rightarrow (r-\Delta<f(r)<t-\Delta<f(t))
\end{equation}
для всех $r, t\in \Sp^1$. Таким образом, $f$ строго возрастает. Определим $\bar{\rho}^1_{\varepsilon}:X_1\times X_1\to \mathbb R^+$ как
\begin{equation}\label{eq4.18}
\bar{\rho}^1_{\varepsilon}(x,y)=
\begin{cases}
0, &\text{если } \, x=y\\
f(\rho^1_{\varepsilon}(x,y)), &\text{если } \, x\neq y.
\end{cases}
\end{equation}
Так как $(X_1,\rho^1_{\varepsilon})\in \mathfrak U$ и $f$~--- строго возрастающая функция, то $\bar{\rho^1_{\varepsilon}}$~--- ультраметрика на $X_1$ и $(X_1,\bar{\rho^1_{\varepsilon}})\in \mathfrak U$. В силу~(\ref{eq4.16}) и ~(\ref{eq4.18}) имеем
$$
|\rho^1_{\varepsilon}(x,y)-\bar{\rho^1_{\varepsilon}}(x,y)|<\Delta.
$$
А так как по~(\ref{eq4.14}) $\Delta\leqslant \frac{\varepsilon}{4}$, то~(\ref{eq4.6}) выполняется для всех $x,y \in X$. Осталось заметить, что~(\ref{eq4.7}) следует непосредственно из~(\ref{eq4.16}) и~(\ref{eq4.18}).

Пусть теперь имеет место случай (ii), $G_d=G[X_1, X_2,...,X_k]$ с $k\geqslant 3$. Для того, чтобы найти ультраметрику $\rho_{\varepsilon}$, для которой $(X,\rho_{\varepsilon})\in \mathfrak U$ и~(\ref{eq4.3}) выполняется при всех $x,y \in X$, достаточно построить ультраметрику $d_{\varepsilon}$ на $X$ так, чтобы граф $G_{d_{\varepsilon}}$ был полным двудольным и неравенство
\begin{equation}\label{eq4.19}
|d(x,y)-d_{\varepsilon}(x,y)|\leqslant \varepsilon
\end{equation}
имело место при всех  $x, y\in X$. После этого искомая $\rho_{\varepsilon}$ строится как в случае (i).

Положим
$$
Y:=\bigcup\limits_{i=2}^k X_i, \quad r:=\max t,\quad  t\in \Sp (X)\backslash \{\diam X\},
$$
где при $\Sp(X)=\{\diam X\}$, считаем $r:=0$.
Пусть $d^{'}\in (r, \diam X)$ и
\begin{equation}\label{eq4.20}
|\diam X-d^{'}|<\varepsilon.
\end{equation}
Определим $d_{\varepsilon}:X\times X\to \mathbb R^+$ правилом
\begin{equation}\label{eq4.21}
d_{\varepsilon}(x,y)=
\begin{cases}
d(x,y), &\text{если } \, d(x,y)\neq \diam X,\\
\diam X, &\text{если } \, x\in X_1 \text{ и } y \in Y \text{ или } x\in Y \text{ и } y\in X_1,\\
d^{'}, &\text{если } \, x, y\in Y \text{ и } d(x,y)=\diam X.
\end{cases}
\end{equation}
Непосредственно из~(\ref{eq4.21}) получаем
$$
(d_{\varepsilon}(x,y)\neq d(x,y))\Rightarrow ((d_{\varepsilon}(x,y)=d^{'} \, \& \, d(x,y)=\diam X)),
$$
что вместе с~(\ref{eq4.20}) даёт ~(\ref{eq4.19}). Убедимся теперь в том, что $d_{\varepsilon}$~--- ультраметрика. Для этого достаточно проверить сильное неравенство треугольника
\begin{equation}\label{eq4.22}
d_{\varepsilon}(x,y)\leqslant \max \{d_{\varepsilon}(x,z), d_{\varepsilon}(z,y)\}
\end{equation}
при всех $x, y, z\in X$.
Так как $d|_{X_1\times X_1}=d_{\varepsilon}|_{X_1\times X_1}$, то~(\ref{eq4.22}) выполнено при $x, y, z\in X_1$. Функция
\begin{equation*}
f:\Sp(X)\to (0,\infty), \quad f(t)=
\begin{cases}
t, &\text{если } \, t \neq \diam X,\\
d^{'}, &\text{если } \, t=\diam X,
\end{cases}
\end{equation*}
строго возрастает, а из~(\ref{eq4.21}) при всех $x, y \in Y$ легко получить равенство
$$
d_{\varepsilon}|_{Y\times Y}(x,y)=f(d|_{Y\times Y})(x,y).
$$
Следовательно $d_{\varepsilon}|_{Y\times Y}$~-- порождает ультраметрику на $Y$ и значит~(\ref{eq4.22}) выполнено при $x,y,z \in Y$. Осталось доказать~(\ref{eq4.22}) при
\begin{equation}\label{eq4.23}
\{x, y, z\}\cap X_1 \neq \varnothing \neq \{x, y, z\}\cap Y.
\end{equation}
Если~(\ref{eq4.23}) выполнено, то из~(\ref{eq4.21}) легко получить равенство
$$
\max\{d_{\varepsilon}(x,z),d_{\varepsilon}(y,z)\}=\diam X.
$$
Кроме того, непосредственно из~(\ref{eq4.21}) следует, что
$$
\max\limits_{x,y\in X}d_{\varepsilon}(x,y) = \diam X.
$$
Таким образом, $d_{\varepsilon}$~--- ультраметрика на $X$. Осталось проверить, что $G_{d_{\varepsilon}}$~--- полный двудольный граф. Это почти очевидно т.к. в соответствии с~(\ref{eq4.21}) $\{x,y\}$~--- диаметральная пара  для $(X, d_{\varepsilon})$ тогда и только тогда, когда $x\in X_1$ и $y\in Y$ или $x\in Y$ и $y\in X_1$.
\end{proof}

Утверждение~\ref{st27} может быть усилено следующим образом.

\begin{statement}
Для любого целого положительного $n$ множество $I_{\mathfrak U}^n:=\{\is (X): X\in \mathfrak U, |X|=n\}$ является плотным открытым подмножеством пространства изометрических типов ультраметрических пространств $(X,d)$ с $|X|\leqslant n$, наделённом метрикой Громова-Хаусдорфа.
\end{statement}

\begin{proof}
Проверим открытость множества $I_{\mathfrak U}^n$. Пусть $(X,d)\in \mathfrak U$, $|X|=n$. Нужно найти такое $\varepsilon$, что для любого ультраметрического $(Y,\rho)$ с $|Y|\leqslant n$ из неравенства
\begin{equation}\label{eq4.24}
d_{GH}(\is (X),\is(Y))<\varepsilon
\end{equation}
следует принадлежность
\begin{equation}\label{eq4.25}
(Y,\rho)\in \mathfrak U
\end{equation}
и равенство $|Y|=n$.  Пусть
\begin{equation}\label{eq4.26}
0<\varepsilon< \frac{1}{4}(\min\limits_{\substack{r,t\in \Sp(X)\\ r\neq t}} |r-t|\wedge \min\limits_{r \in \Sp(X)}r).
\end{equation}
В соответствии с определением~\ref{def4.1} найдётся метрическое пространство $(Z,d_Z)$ такое, что для некоторых $X^1, Y^1\subseteq Z$ выполняются равенства $\is(X^1)=\is(X)$, $\is(Y^1)=\is(Y)$ и
\begin{equation}\label{eq4.26*}
X^1\subseteq \bigcup\limits_{y\in Y^1} O_{\varepsilon} (y),
\end{equation}
где $O_{\varepsilon}(y)=\{z\in Z:d_Z(z,y)<\varepsilon\}$. Заметим, что, если $x_1, x_2$~--- разные точки из $X^1$, a $y_1$, $y_2$~--- точки из $Y^1$ такие, что $x_i \in O_{\varepsilon} (y_i)$, $i=1, 2$, то $y_1\neq y_2$. Действительно, если $y_1=y_2$, то
$$
d_Z(x_1,x_2)<d_Z(x_1,y_1)+d_Z(y_1,x_2)<2\varepsilon,
$$
что противоречит~(\ref{eq4.26}). Следовательно $|X^1|\leqslant |Y^1|$. А так как $|X^1|=|X|=n\geqslant |Y|=|Y^1|$, то имеем равенство  $|Y|=n$. Осталось доказать~(\ref{eq4.25}). Так как $(X,d)\in \mathfrak U$, то для этого нужно проверить равенство
$|\Sp(Y^1)|=|\Sp(X^1)|$.  Неравенство $|\Sp(Y^1)|\leqslant |\Sp(X^1)|$ вытекает из неравенства Гомори-Ху, поэтому достаточно установить неравенство $|\Sp(Y^1)|\geqslant|\Sp(X^1)|$. Последнее неравенство эквивалентно выполнению для всех $y_1, y_2, y_3, y_4 \in Y^1$ импликации
\begin{equation}\label{eq4.27}
(d_Z(y_1,y_2)=d_Z(y_3,y_4))\Rightarrow (d_Z(x_1,x_2)=d_Z(x_3,x_4)),
\end{equation}
где $x_i$~--- единственная точка из $X^1$, лежащая в шаре $O_{\varepsilon}(y_i)$, $i=1, 2, 3, 4$. (Эта единственность фактически уже доказана выше). Пусть в~(\ref{eq4.27}) $d_Z(y_1,y_2)=d_Z(y_3,y_4)$, но  $d_Z(x_1,x_2)>d_Z(x_3,x_4)$. Тогда, используя неравенство треугольника и~(\ref{eq4.26}), находим
\begin{gather*}
\min\limits_{\substack{r,t\in \Sp\\ r\neq t}} |r-t|\leqslant d_Z(x_1,x_2)-d_Z(x_3,x_4)\\
\leqslant d_Z(x_1,y_1)+d_Z(y_1,y_2)+d_Z(y_2,x_2)\\ -(d_Z(y_3,y_4)-d_Z(x_3,y_3)-d_Z(x_4,y_4))=\sum\limits_{i=1}^n d(x_i,y_i)<4\varepsilon,
\end{gather*}
что противоречит~(\ref{eq4.26}). Аналогично можно показать, что и неравенство $d_Z(x_3,x_4)>d_Z(x_1,x_2)$ противоречит~(\ref{eq4.26}). Таким образом, импликация~(\ref{eq4.27}) выполняется для всех $y_1, y_2, y_3, y_4 \in Y^1$.

Покажем, что $I_{\mathfrak{U}}^n$~--- плотное множество. Для этого нужно установить, что для любого $\varepsilon>0$ и любого конечного ультраметрического $(Y,\rho)$ с $|Y|\leqslant n$ существует $(X,d)\in \mathfrak U$ такое, что
\begin{equation}\label{eq4.28}
|X|=n \, \text{ и } \, d_{GH}(\is X, \is Y)<\varepsilon.
\end{equation}
Если $|Y|=n$, то  существование $(X,d)\in \mathfrak U$, удовлетворяющего~(\ref{eq4.28}), следует из утверждения~\ref{st27}. Пусть $m:=n-|Y|>0$. Выберем $y_1,...,y_m$ так, что
$$
\{y_1,...,y_m\}\cap Y=\varnothing
$$
и зафиксируем точку $y_0\in Y$. Пусть $\varepsilon$ удовлетворяет неравенству \begin{equation}\label{eq4.29}
0<\varepsilon<\inf\{\rho(x,y):x,y\in Y, x\neq y\},
\end{equation}
где при $|Y|=1$ полагаем  $\inf \varnothing =+\infty$. Пусть $Z:=Y\cup\{y_1,...,y_m\}$. Определим на $Z$ ультраметрику $\rho_{\varepsilon}$ как ``$\frac{\varepsilon}{2}$-раздутие'' точки $y_0$,
\begin{equation}\label{eq4.30}
\rho_{\varepsilon}(x,y)=
\begin{cases}
\rho(x,y), &\text{если } \, x,y\in Y,\\
\rho(x,y_0), &\text{если } \, x\in Y \text{ и } y \in \{y_0,...,y_m\} \\
&\text{или } y\in Y \text{ и } x \in \{y_0,...,y_m\},\\
\frac{\varepsilon}{2}(1-\delta_{x,y}), &\text{если } \, x, y\in \{y_0,...,y_m\},
\end{cases}
\end{equation}
где $\delta_{x,y}$~--- символ Кронекера $\delta_{x,y}=1$ при $x=y$ и $\delta_{x,y}=0$ при $x\neq y$. Используя~(\ref{eq4.29}) и~(\ref{eq4.30}) легко показать, что $(Z,\rho_{\varepsilon})$~--- ультраметрическое пространство, $|Z|=n$ и
\begin{equation}\label{eq4.31}
d_{GH}(\is Y, \is Z)<\frac{\varepsilon}{2}.
\end{equation}
По утверждению~\ref{st27} существует $(X,d)\in \mathfrak U$ такое, что $|X|=n$, $d_{GH}(\is X, \is Y)<\frac{\varepsilon}{2}$. Отсюда,~(\ref{eq4.31}) и неравенства треугольника для $d_{GH}$ следует~(\ref{eq4.28}).
\end{proof}

\normalsize

\textbf{Е. A. Петров}

Институт прикладной математики и механики

НАН Украины, г. Донецк.

Ул. Розы Люксембург 74, Донецк, Украина,  индекс 83114.

eugeniy.petrov@gmail.com

\bigskip

\textbf{A. A. Довгошей}

Институт прикладной математики и механики

НАН Украины, г. Донецк.

Ул. Розы Люксембург 74, Донецк, Украина,  индекс 83114.

aleksdov@mail.ru

\end{document}